\documentclass[12pt, reqno]{amsart}
\usepackage{amsmath, amsthm, amscd, amsfonts, amssymb, graphicx, color}
\usepackage[bookmarksnumbered, colorlinks, plainpages]{hyperref}

\newtheorem{theorem}{Theorem}[section]
\newtheorem{lemma}[theorem]{Lemma}
\newtheorem{proposition}[theorem]{Proposition}

\theoremstyle{definition}
\newtheorem{definition}[theorem]{Definition}
\newtheorem{example}[theorem]{Example}

\theoremstyle{remark}
\newtheorem{remark}[theorem]{Remark}
\numberwithin{equation}{section}

\newcommand{\N}{{\mathbb{N}}}
\newcommand{\cstar}{\mbox{$C^*$}}
\newcommand{\Um}{\textstyle\mathop{\textrm{Um}}}
\newcommand{\Lg}{\textstyle\mathop{\textrm{Lg}}}
\newcommand{\GL}{\textstyle\mathop{\textrm{GL}}}
\newcommand{\Gen}{\textstyle\mathop{\textrm{Gen}}}
\newcommand{\sr}{\mathop{\textrm{sr}}}
\newcommand{\tsr}{\mathop{\textrm{tsr}}}
\newcommand{\Bsr}{\mathop{\textrm{Bsr}}}
\renewcommand{\span}{\mathop{\textrm{span}}}
\newcommand{\p}[2]{\langle #1,#2\rangle}
\newcommand{\Sum}{\mathop{\mbox{$\sum$}}}
\newcommand{\SUM}{\displaystyle\mathop{\mbox{$\sum$}}}

\begin{document}
\setcounter{page}{1}

\title[The stable rank of $\cstar$-modules]{The stable rank of $\cstar$-modules}

\author[M. Achigar]{Mauricio Achigar}

\address{Facultad de Ciencias, Igu\'{a} 4225, CP 11400, Montevideo, Uruguay.}
\email{\textcolor[rgb]{0.00,0.00,0.84}{achigar@cmat.edu.uy; mauricio.achigar3@gmail.com}}

%\dedicatory{This paper is dedicated to Professor ABCD}



\date{Received: xxxxxx; Revised: yyyyyy; Accepted: zzzzzz. \newline\indent  Partially supported by Proyecto Fondo Clemente Estable FCE2007\!\_731.
}

\begin{abstract}
We prove equality between the Topological Stable Rank and the Bass Stable Rank for finitely generated projective left modules over a unital \cstar-algebra. In order to do so, the concept of Stable Rank of a Hilbert module is introduced.
\end{abstract} \maketitle

\section{Introduction and preliminaries}

In the mid 1960s, H. Bass introduced the concept of Stable Rank of a ring $A$, now refered to as Bass Stable Rank and denoted by $\Bsr(A)\in\N$. In the late 1970s, R. B. Warfield extends this concept defining the Bass Stable Rank for modules over rings. Later, in \cite{Rie}, M. A. Rieffel introduced the notion of Topological Stable Rank for a Banach algebra $A$, $\tsr(A)\in\N$, as well as for Banach modules over unital Banach algebras. In this work, Rieffel shows that $\Bsr(A)\leq\tsr(A)$ holds for unital Banach Algebras and that $\Bsr(V)\leq\tsr(V)$ holds for finitely generated modules $V$ over unital \cstar-algebras. In \cite{HV}, R. H. Herman and L. N. Vaserstein prove that $\Bsr(A)=\tsr(A)$ for a unital \cstar-algebra $A$.

In this article we show that  $\Bsr(V)=\tsr(V)$  for finitely generated projective modules $V$ over unital \cstar-algebras, using similar techniques to the ones presented in \cite{HV}. In order to generalize Herman-Vaserstein's theorem, we introduce the concept of Stable Rank for a Hilbert module. This definition is inspired by the works of Ara and Goodearl \cite{AG} and Blackadar \cite{Bla}.

\section{Types of stable ranks}

\begin{definition}
Let $_AV$ be a left module over a ring $A$. The set of $n$-generators of $V$ is defined as
   $$\Gen_n(V)=\{(x_1,\ldots,x_n)\in V^n: A\cdot x_1+\cdots+A\cdot x_n=V\}.$$
When $n=1$ we simply write $\Gen(V)=\Gen_1(V)$. We say that an $(n+1)$-generator $(x_1,\ldots,x_n,y)\in\Gen_{n+1}(V)$ is {\it reducible} if there exist $a_1,\ldots,a_n\in A$ such that $(x_1+a_1\cdot y,\ldots,x_n+a_n\cdot y)\in\Gen_{n}(V)$.
\end{definition}

\begin{remark}
The column space $V^n$ can be viewed as a left module over the matrix ring $ M_n(A)$. In that case we have
   $$\Gen_n(_AV)=\Gen(_{M_n(A)}V^n).$$
\end{remark}

\begin{definition}
Let $_AV$ be a left module over a ring $A$. The {\it Bass Stable Rank} of $V$, denoted $\Bsr(V)$, is defined as the least $n\in\N$ ($n\geq 1$) such that every $(n+1)$-generator $(x_1,\ldots,x_n,y)\in\Gen_{n+1}(V)$ is reducible.
\end{definition}

When $V={_A}A$ and $A$ has a unit, this definition becomes definition 2.1 in \cite{Rie} and it's due to Bass. For arbitrary $V$ and unital $A$, the above definition is equivalent to definition \cite[Definition 9.1]{Rie} of Bass stable rank for modules, introduced by Warfield.

\begin{definition} (\cite[Definition 9.3]{Rie}) Let $_AV$ be a left Banach module over a Banach algebra $A$. The {\it Topological Stable Rank} of $V$ is defined as
   $$\tsr(V)=\min\{n\in\N:\Gen_n(V)\text{ is dense in }V^n\}.$$	
\end{definition}

\section{Stable rank for Hilbert modules}

\begin{definition}
Given a right Hilbert module $X_B$ over a unital $C^*$-algebra $B$, we consider
   $$\Um_n(X)=\{(x_1,\ldots,x_n)\in X^n: \textstyle\sum_k\p{x_k}{x_k}\in\GL(B)\}.$$
An $n$-tuple in $\Um_n(X)$ is called {\it unimodular} tuple. If $n=1$ we write $\Um(X)=\Um_1(X)$.
\end{definition}

\begin{remark}\label{UmXn} The column space $X^n$ can be viewed as a right $C^*$-module over $B$, and in this case we have
   $$\Um_n(X_B)=\Um(X_B^n).$$
\end{remark}

\begin{definition}
Let $X_B$ be a right Hilbert module over a unital $C^*$-algebra $B$. We define the {\it Stable Rank} of $X_B$ as
 $$\sr(X_B)=\min\{n\in\N:\Um_n(X)\text{ is dense in }X^n\}.$$
\end{definition}

Note that if $k\geq\sr(X_B)$ then $\Um_k(X)$ is dense in $X^k$.

\begin{remark}
For a unital \cstar-algebra $B$, taking $X_B=B_B$ in the previous definition we recover Rieffel's definition of (left) topological stable rank of $B$  \cite[Definition 1.4]{Rie}. Indeed, lemma \ref{lema1} or remark \ref{rk1} can be used to see that  $\Um_n(B_B)=\Lg_n(B)$, the later being the set of left $n$-generators of $B$ considered in \cite{Rie}. Then, $\sr(B_B)=\tsr(B)$ ($=\Bsr(B)$).
\end{remark}

\begin{remark}
For projections $p,q$ in a \cstar-algebra $A$, Blackadar (\cite{Bla}) considered the set
   $$\Lg_{(p,q)}(A)=\{x\in pAq:{\exists}\;y\in qAp\text{ such that }yx=q\}.$$
and used the condition of $\Lg_{(p,q)}(A)$ being dense in $pAq$.
Taking $X$ as the skew corner $X=pAq$ and $B=qAq$ we have, by lemma \ref{lema1}, that $\Lg_{(p,q)}(A)=\Um(X_B)$, and $\Lg_{(p,q)}(A)$ is dense in $pAq$ if and only if $\sr(X_B)=1$.
\end{remark}

\begin{example} Let $A$ be a unital \cstar-algebra and consider the set $M_{n\times m}(A)$ as a right \cstar-module over $M_m(A)$ with formal matrix operations. The stable rank of $M_{n\times m}(A)$ is
\begin{equation}\label{srMnm}
  \sr(M_{n\times m}(A))
  =\left\lceil\frac{\sr(A)+m-1}n\right\rceil.
\end{equation}
This expression extends the well-known formula for $\sr(M_n(A))$ (\cite[theorem 6.1]{Rie}).
\end{example}

\begin{proof} Firstly, by lemma \ref{lema1} we know that  $\Um(M_{n\times m}(A))$ is the set of left invertible $n\times m$ matrices over $A$. Then, by \cite[Corollary 4.3]{Bla} we have
\begin{equation*}
  \Um(M_{(n+1)\times 1}(A))
  \textit{ is dense}\quad\textit{iff}\quad
  \Um(M_{(n+k)\times k}(A))\textit{ is dense},
\end{equation*}
where ``$\Um(M_{r\times s}(A))$ dense'' means dense  in $M_{r\times s}(A)$. Equivalently,
\begin{equation}\label{bla4.3}
  \Um(M_{r\times s}(A))
  \textit{ is dense}\quad\textit{iff}\quad
  \Um(M_{(r-s+1)\times 1}(A))\textit{ is dense},
  \quad\textit{for }r\geq s.
\end{equation}

Secondly, we can realize $M_{n\times m}(A)$ as a skew corner of $M_{\bar n}(A)$ for $\bar n$ large in the following way: $M_{n\times m}(A)\cong pM_{\bar n} (A)q$ for $p,q\in M_{\bar n} (A)$ diagonal projections of ranks $n$ and $m$, respectively. Then, by \cite[Proposition 3.2.iii]{Bla} we have
\begin{equation}\label{bla3.2}
  \textit{If }\Um(M_{n\times m}(A))
  \textit{ is dense then }n\geq m.
\end{equation}

For $k\in\N$, we have $k\geq\sr(M_{n\times m}(A))$ iff $\Um_k(M_{n\times m}(A))$ is dense. Identifying the column space $M_{n\times m}(A)^k$ with $M_{nk\times m}(A)$ we have
$\Um_k(M_{n\times m}(A))
=\Um(M_{n\times m}(A)^k)
=\Um(M_{nk\times m}(A))$.
If $\Um(M_{nk\times m}(A))$ is dense then $nk\geq m$, by \eqref{bla3.2}, and $\Um(M_{(nk-m+1)\times1}(A))$ is dense, by \eqref{bla4.3}.  Therefore $nk-m+1\geq\sr(A)$ by definition of $\sr (A)$, and then $k\geq\frac{\sr(A)+m-1}n$. Conversely, the inequalities imply $nk\geq m$ and the density of  $\Um(M_{(nk-m+1)\times1}(A))$, then by \eqref{bla4.3} $\Um(M_{nk\times m}(A))$ $(=\Um_k(M_{n\times m}(A)))$ is dense, and finally $k\geq\sr(M_{n\times m}(A))$. Thus we have shown that $k\geq\sr(M_{n\times m}(A))$ {\it iff} $k\geq\frac{\sr(A)+m-1}n$, which is equivalent to \eqref{srMnm}.
\end{proof}

\begin{lemma}\label{lema1}
Let $X_B$ be a unital Hilbert module. For $x\in X$, the following  are equivalent:
\begin{enumerate}
\item[(a)] $\p xx\in\GL(B)$.
\item[(b)] There exists $y\in X$ such that $\p yx=1$.
\item[(b*)] There exists $y\in X$ such that $\p xy=1$.
\item[(c)] There exists $y\in X$ such that $\p yx\in\GL(B)$.
\item[(c*)] There exists $y\in X$ such that $\p yx\in\GL(B)$.
\end{enumerate}
\end{lemma}

\begin{proof}
It suffices to show that (a) is equivalent to (b). If $\p xx\in\GL(B)$ let $b=\p xx$ and $y=x(b^{-1})^*$. Then $\p yx=b^{-1}\p xx=1$. On the other hand, if $\p yx=1$ we have $1=\p yx^*\p yx\leq\|y\|^2\p xx$, so that $\p xx\in\GL(B)$.
\end{proof}

\begin{remark}\label{rk1} Applying the previous lemma to $X^n$  and using \ref{UmXn} we obtain different expressions for $\Um_n(X)$. For example, using (b) we have
  $$\Um_n(X)=\{(x_1,\ldots,x_n)\in X^n: \exists\ y_1,\ldots,y_n\in X\
  \slash\ \textstyle\sum_k\p{y_k}{x_k}=1\}.$$
\end{remark}

\begin{lemma}\label{full}
Let $X_B$ be a unital Hilbert module. Then $X_B$ is full if and only if there exists $n\in\N$ such that $\Um_n(X)\neq\varnothing$.
\end{lemma}

\begin{proof}
The module $X_B$ is full if and only if the $*$-ideal $J=\span(\p XX)$ is dense in $B$, iff $J\cap\GL(B)\neq\varnothing$, iff $1\in J$ ($=B$), iff there exists $n\in\N$ and
$x_1,\ldots,x_n,y_1,\ldots,y_n\in X$ such that $\sum_k\p{x_k}{y_k}=1$, iff there exists $n\in\N$ and $x,y\in X^n$ such that $\p xy=1$, iff there exists $n\in\N$ such that $\Um_n(X)\neq\varnothing$.
\end{proof}

Notice that if $X_B$ is not full then $\Um_n(X)=\varnothing$ for all $n\in\N$, and therefore $\sr(X)=\infty$. Thus, throughout this paper we shall consider $X_{B}$ to be full.

\begin{definition}
 A {\it\cstar-correspondence} is a right Hilbert module $X_B$ equipped with a left action $A\to\mathcal{L}(X_B)$ of a \cstar-algebra $A$ by adjointable operators. When $B$ has a unit we say that the correspondence is {\it right-unital}.
\end{definition}

\begin{lemma} Let $_AX_B$ be a full and right-unital $C^*$-correspondence. Then
   $$\Gen_n(_AX)\subseteq\Um_n(X_B)\ \ \forall\, n\in\N.$$
\end{lemma}

\begin{proof}
Since $\Gen_n(_AX)=\Gen(_{M_n(A)}X^n)$, $\Um_n(X_B)=\Um(X_B^n)$, and the $C^*$-correspondence $_{M_n(A)}X^n_B$ is full and right-unital, we may assume $n=1$. As $X_B$ is full there exist $x_1,\ldots,x_r,y_1,\ldots,y_r\in X$ such that $\sum_k\p{x_k}{y_k}=1$. If $x\in\Gen(_AX)$ let $a_1,\ldots,a_r\in A$ be such that $x_k=a_k\cdot x$, for $k=1,\ldots,r$. Then,
   $$\textstyle1=\sum_k\p{x_k}{y_k}
     =\sum_k\p{a_k\cdot x}{y_k}
     =\sum_k\p{x}{a_k^*\cdot y_k}
     =\p{x}{\sum_ka_k^*\cdot y_k},$$
so $x\in\Um(X_B)$.
\end{proof}

\begin{definition}
A vector space $X$ is said to be a {\it\cstar-bimodule} whenever it is equipped with compatible left and right Hilbert module structures over \cstar-algebras $A$ and $B$, respectively.
\end{definition}

\begin{lemma}
Let $_AX_B$ be a right-full and right-unital $C^*$-bimodule. Then
   $$\Um_n(X_B)\subseteq\Gen_n(_AX)\ \ \forall\, n\in\N.$$
\end{lemma}

\begin{proof}
Since $\Gen_n(_AX)=\Gen(_{M_n(A)}X^n)$, $\Um_n(X_B)=\Um(X_B^n)$ and the $C^*$-bimodule $_{M_n(A)}X^n_B$ is right-full and right-unital, we may assume $n=1$. Given $x\in\Um(X)$, let $y\in X$ be such that $\p yx_R=1$. Then, for all $z\in X$ we have $z=z\cdot1=z\cdot\p yx_R=\p zy_L\cdot x$. Then $x\in\Gen(X)$.
\end{proof}

\begin{proposition}
Let $_AX_B$ be a right-full and right-unital $C^*$-bimodule. Then
   $$\Um_n(X_B)=\Gen_n(_AX)\ \ \forall\, n\in\N\quad\text{and}\quad \tsr(_AX)=\sr(X_B).$$
\end{proposition}

\section{Stable rank Inequality for $C^*$-modules}

\cite[Proposition 9.7]{Rie} says that $\Bsr(V)\leq\tsr(V)$ for a finitely generated projective left module $_AV$ over a unital $C^*$-algebra $A$. Inspired by this, we prove a similar result in the $C^*$-module context, namely, \emph{if $X$ is a right-full and right-unital $C^*$-bimodule, then $\Bsr(X)\leq\tsr(X)$.}

\begin{lemma} (Warfield Condition, \cite[Propositions 2.2 and 9.2]{Rie}) Let $_AX_B$ be a right-full and right-unital $C^*$-bimodule and $x_1,\ldots,x_{n+1}\in X$. The following conditions are equivalent:
\begin{enumerate}
\item[(a)] $\exists\ a_1,\ldots,a_n\in A$ $/$ $(x_1+a_1\cdot x_{n+1},\ldots,x_n+a_n\cdot x_{n+1})\in\Um_n(X)$.
\item[(b)] $\exists\ y_1,\ldots,y_{n+1}\in X$ $/$ $\SUM_{k=1}^{n+1}\p {y_k}{x_k}_R=1$ and $(y_1,\ldots,y_n)\in\Um_n(X)$.
\end{enumerate}
\end{lemma}

\begin{proof}
If $(x_1+a_1\cdot x_{n+1},\ldots,x_n+a_n\cdot x_{n+1})\in\Um_n(X)$, there exist $y_1,\ldots,y_n\in X$ such that $\SUM_{k=1}^n\p{y_k}{x_k+a_k\cdot x_{n+1}}_R=1$. Then we have $(y_1,\ldots,y_n)\in\Um_n(X)$ by \ref{lema1}, and
   $$\Sum_{k=1}^n\p{y_k}{x_k}_R+\p{\Sum_{k=1}^na_k^*\cdot y_k}{x_{n+1}}_R=1.$$
Taking $y_{n+1}=\SUM_{k=1}^na_k^*\cdot y_k$ we obtain (b).

Conversely, if condition (b) holds, as $(y_1,\ldots,y_n)\in\Um_n(X)$ there exist $z_1,\ldots,z_n\in X$ such that $\SUM_{k=1}^n\p{y_k}{z_k}_R=1$. Let $a_k=\p{z_k}{y_{n+1}}_L$, for $k=1,\ldots,n$. Then
\begin{equation}\label{eqn1}
   \SUM_{k=1}^n\p{y_k}{x_k+a_k\cdot x_{n+1}}_R
   =\SUM_{k=1}^n\p{y_k}{x_k}_R+\p{\SUM_{k=1}^n a_k^*\cdot y_k}{x_{n+1}}_R.
\end{equation}
Now, we have
\begin{eqnarray*}
   \SUM_{k=1}^n a_k^*\cdot y_k
   =\SUM_{k=1}^n \p{y_{n+1}}{z_k}_L\cdot y_k
   =\SUM_{k=1}^ny_{n+1}\cdot\p{z_k}{y_k}_R\\
   =y_{n+1}\cdot\SUM_{k=1}^n\p{z_k}{y_k}_R
   =y_{n+1}\cdot1=y_{n+1}.
\end{eqnarray*}
Then, the right-hand side of equation \ref{eqn1} equals 1 by (b), and (a) holds.
\end{proof}

The proof of the following proposition is analogous to that of \cite[Theorem 2.3]{Rie}.

\begin{proposition}(\cite[Proposition 9.7]{Rie})\label{Bsrleqtsr}
Let $_AX_B$ be a right-full and right-unital $C^*$-bimodule. Then
   $$\Bsr(X)\leq\tsr(X).$$
\end{proposition}

\begin{proof}
Let $n=\tsr(X)=\sr(X)$. Given $(x_1,\ldots,x_{n+1})\in\Gen_{n+1}(_AX)=\Um_{n+1}(X_B)$ consider $z_1,\ldots,z_{n+1}\in X$ such that $\SUM_{k=1}^{n+1}\p{z_k}{x_k}_R=1$. For $k=1,\ldots,n$, pick perturbations $\bar z_k\simeq z_k$ with $(\bar z_1\ldots,\bar z_n)\in\Um_n(X)$ so that
    $$d^*:=\p{\bar z_1}{x_1}_R+\cdots+\p{\bar z_n}{x_n}_R+\p{z_{n+1}}{x_{n+1}}_R\in\GL(B).$$
Then, taking $y_1=\bar z_1\cdot d^{-1},\ldots,y_n=\bar z_n\cdot d^{-1}, y_{n+1}=z_{n+1}\cdot d^{-1}$ we have $(y_1,\ldots,y_n)=(\bar z_1,\ldots,\bar z_n)\cdot d^{-1}\in\Um_n(X)$ and $\SUM_{k=1}^{n+1}\p{y_k}{x_k}_R=1$. By the previous lemma $(x_1,\ldots,x_{n+1})$ is reducible, and then $\Bsr(X)\leq n$.
\end{proof}

\section{Herman-Vaserstien theorem for \cstar-modules}
Herman-Vaserstein theorem states that for a unital \cstar-algebra $A$, $\tsr(A)\leq\Bsr(A)$. In this section we obtain $\tsr(X)\leq\Bsr(X)$ for a right-full and right-unital Hilbert bimodule $X$.

\begin{lemma}\label{lema_HV}
Let $X_B$ be a full and unital Hilbert module. Given $x_1,\ldots,x_n$, $u_1,\ldots,u_r\in X$ such that $\Sum_k\p{u_k}{u_k}=1$, and $\varepsilon>0$, let $b_0=\sum_i\p{x_i}{x_i}$, $b=(1-\frac{b_0}{\varepsilon})^+$ and $y_k=u_k\cdot b$, for $k=1,\ldots,r$. Then $(x_1,\ldots,x_n,y_1,\ldots,y_r)\in\Um_{n+r}(X_B)$.
\end{lemma}

\begin{proof}
Let $x=(x_1,\ldots,x_n)$, $u=(u_1,\ldots,u_r)$ and $y=(y_1,\ldots,y_r)$, then $y=u\cdot b$ and $b_0=\p xx$. Consider the commutative $C^*$-subalgebra $B_0:=C^*(1,b_0)\subseteq B$. Let $c\in B$ the element given by
   $$c=\p yy=\p{u\cdot b}{u\cdot b}=b^*\p uu b =b^*b=[(1-\frac{b_0}{\varepsilon})^+]^2.$$
Consequently $c$ and $b_0$ belongs to $B_0^+$ and do not have common roots. Therefore
   $$\p{(x,y)}{(x,y)}=\p xx + \p yy=b_0+c\in\GL(B_0)\subseteq{\GL(B)}.$$
That is, $(x,y)=(x_1,\ldots,x_n,y_1,\ldots,y_r)\in\Um_{n+r}(X_B)$.
\end{proof}

\begin{theorem}\label{HV}
Let $_AX_B$ be a right-full and right-unital $C^*$-bimodule. Then
   $$\Bsr(X)=\tsr(X).$$
\end{theorem}

\begin{proof}
By proposition \ref{Bsrleqtsr} it suffices to show $\Bsr(X)\geq\tsr(X)$. Suppose $\Bsr(_AX)=n$ and let $x=(x_1,\ldots,x_n)\in X^n$, $\varepsilon>0$ be given. As $X$ is right-full and right-unital, by \ref{full} there exists $u=(u_1,\ldots,u_r)\in\Um_r(X)$ for suitable $r\in\N$. Replacing $u$ with $u\cdot\p uu_R^{-1/2}$, we may suppose $\p uu_R=1$. Taking $b_0$, $b$ and $y$ as in lemma \ref{lema_HV} we have that $(x,y)\in\Um_{n+r}(X_B)=\Gen_{n+r}(_AX)$. Then, as $\Bsr(_AX)=n$, the generator $(x,y)$ can be reduced $r$ times to an $n$-generator. Therefore, there exists $a\in M_{n\times r}(A)$ such that $x+a\cdot y\in\Gen_n(_AX)=\Um_n(X_B)$. Let
  $$k>\frac{\|a\|}{\varepsilon},\quad d=1+kb\in B^+\cap\GL(B)\quad\text{and}$$
  $$\quad x'=(x+a\cdot y)\cdot d^{-1}\in\Um_n(X)=\Gen_n(X),$$
where $\|a\|$ is the norm of $a$ as a $B$-adjuntable operator $a\colon X^r\to X^n$.

We have $x-x'=(x\cdot d-x-a\cdot y)\cdot d^{-1}=(x\cdot kb-a\cdot y)\cdot d^{-1}$ and
\begin{equation}\label{1}
  \|x-x'\|\leq\|x\cdot kbd^{-1}\|+\|a\cdot y\cdot d^{-1}\|.
\end{equation}
As $b_0=\p xx_R$, we have
   $$|x\cdot kbd^{-1}|_R^2    =(kbd^{-1})^*\p{x}{x}_R(kbd^{-1})=(kbd^{-1})^*b_0(kbd^{-1}).$$
Now, as $b=(1-\frac{b_0}{\varepsilon})^+$, we have $d=1+kb\in C^*(1,b_0)\cong C(T)$ which is commutative. Therefore $kbd^{-1}=kb(1+kb)^{-1}=b(\frac1k+b)^{-1}\leq 1$ in $C(T)$ and consequently $(kbd^{-1})^*b_0(kbd^{-1})\leq b_0$. Moreover, if $b_0(t)>\varepsilon$ for suitable $t\in T$, then $b(t)=0$, because $b=(1-\frac{b_0}{\varepsilon})^+$. Hence $(kbd^{-1})^*b_0(kbd^{-1})\leq\varepsilon$ and
\begin{equation}\label{bound1}
   \|x\cdot kbd^{-1}\|=\||x\cdot kbd^{-1}|_R^2\|^{1/2}\leq\sqrt{\varepsilon}.
\end{equation}

On the other hand, since $y=u\cdot b$ we have
\begin{equation}\label{bound2}
   \|a\cdot y\cdot d^{-1}\|
   =\|\frac ak\cdot u\cdot kbd^{-1}\|
   \leq\frac{\|a\|}k\|u\|\|kbd^{-1}\|
   <\varepsilon,
\end{equation}
where we have used that $\|a\|/k<\varepsilon$, $\|u\|=1$ and $\|kbd^{-1}\|\leq1$.

Thus we can estimate \ref{1} using equations \ref{bound1} and \ref{bound2} to get
   $$\|x-x'\|<\sqrt{\varepsilon}+\varepsilon.$$
Then, $x'$ can be taken arbitarlily close to $x$ and $x'\in\Gen_n(X)$. Therefore, $\Gen_n(X)$ is dense and $\tsr(X)\leq n$.
\end{proof}

\begin{remark} If $_AX$ is a finitely generated projective module over a unital \cstar-algebra $A$ we can make it into a right-full and right-unital $C^*$-bimodule in the following way. The module $_AX$ is a direct summand of $A^n$ for suitable $n\in\N$, and is therefore the range of a (selfadjoint) projection $p\in M_n(A)$. Then we have $_AX$ as the submodule $_A(A^np)$ of $_AA^n$. As we actually have a Hilbert $A-M_n(A)$ bimodule structrue on $A^n$ (thinking of $A^n$ as a row space and using the usual matrix operations) we obtain, by restriction, an $A-pM_n(A)p$ \cstar-bimodule $_A(A^np){_{pM_n(A)p}}$, which is right-full and right-unital.
\par Combining this construction with theorem \ref{HV} we have that $\Bsr(X)=\tsr(X)$ for every finitely generated projective left module over a unital \cstar-algebra.
\end{remark}

{\bf Acknowledgement.} The author wishes to thank his friend Janine Bachrachas for her help editing this article.

\bibliographystyle{amsplain}

\end{document}